\documentclass[11pt,a4paper,reqno]{amsart}

\usepackage[T1]{fontenc}
\usepackage[utf8]{inputenc}
\usepackage{lmodern}
\usepackage[a4paper,left=28mm,right=28mm,top=24mm,bottom=24mm,headheight=13pt,headsep=6mm]{geometry}
\usepackage{amsmath,amssymb,amsthm,mathtools}
\usepackage{microtype}
\usepackage{enumitem}
\usepackage[hidelinks]{hyperref}
\numberwithin{equation}{section}

\hypersetup{
  pdftitle={Regular Bundles on Orbifolds: A Short Proof of Presentability},
  pdfauthor={Enrique Becerra and Ernesto Lupercio},
  pdfsubject={Regular bundles and quotient presentations of smooth orbifolds},
  pdfkeywords={orbifold, regular representation, vector bundle, Adams operation, global quotient, Lie groupoid}
}

\newtheorem{theorem}{Theorem}[section]
\newtheorem{lemma}[theorem]{Lemma}
\newtheorem{proposition}[theorem]{Proposition}
\newtheorem{corollary}[theorem]{Corollary}
\theoremstyle{remark}
\newtheorem{remark}[theorem]{Remark}
\theoremstyle{plain}
\newtheorem*{maintheorem}{Theorem A}
\newtheorem*{maincorollary}{Corollary B}

\newcommand{\X}{\mathcal X}

\newcommand{\C}{\mathbb C}
\newcommand{\Vect}{\operatorname{Vect}}
\newcommand{\Res}{\operatorname{Res}}
\newcommand{\id}{\operatorname{id}}

\newcommand{\Fr}{\operatorname{Fr}}

\title[Regular bundles on orbifolds]
  {Regular Bundles on Orbifolds: A Short Proof of Presentability}

\author{Enrique Becerra}
\address{Departamento de Matem\'aticas, CINVESTAV,
07360 Ciudad de M\'exico, Mexico}
\email{ebecerra@math.cinvestav.mx}

\author{Ernesto Lupercio}
\address{Departamento de Matem\'aticas, CINVESTAV,
07360 Ciudad de M\'exico, Mexico}
\email{lupercio@math.cinvestav.mx}

\date{}

\begin{document}

\begin{abstract}
Let $\X$ be a Hausdorff second-countable smooth orbifold without boundary,
with bounded dimension and stabilizer orders. We construct a smooth complex
vector bundle whose fibre at each point is a positive multiple of the
regular representation of the full stabilizer. Its rank is given by an
explicit formula depending only on the two bounds. Unitary frames give a
presentation of $\X$ as the quotient of a smooth manifold by a locally free
action of a unitary group. The manifold is compact exactly when the coarse space
of $\X$ is compact. The construction applies to noncompact and ineffective
orbifolds. The proof uses Adams operations to cancel the obstructions to
extending regular bundles over a locally finite triangulation. Multiplicity
estimates and the connectivity of complex Stiefel manifolds realize the
resulting virtual classes by actual bundles.
\end{abstract}

\maketitle

\section{Introduction}\label{sec:introduction}

An orbifold is \emph{presentable} if it is equivalent to $[M/K]$, where a
compact Lie group $K$ acts smoothly, properly, and locally freely on a
manifold $M$. For an effective orbifold of fixed dimension, the orthonormal
frame bundle gives such a presentation. Ineffective stabilizer elements act
trivially on the tangent bundle and fix its orthonormal frames.

For an orbifold $\X$, write $|\X|$ for its coarse space and $G_x$ for
the full stabilizer at $x$.
A complex bundle $E$ on $\X$ is \emph{regular} if each fibre $E_x$ is a
positive multiple of the regular representation $\C[G_x]$ of the full
stabilizer. A regular bundle is faithful, so stabilizers act freely on its
unitary frames. We use the compatibility of regular representations under
subgroup restriction to construct such a bundle over successive skeleta,
with a rank bound independent of the number of cells.

For integers $n\geq0$ and $B\geq1$, put
$L=\operatorname{lcm}(1,\ldots,B)$ and
$t(n)=\max\{0,\lfloor(n-1)/2\rfloor\}$. Set $R(n,1)=1$ and, for $B\geq2$,
\begin{equation}\label{eq:rank}
 R(n,B)=L\prod_{s=1}^{t(n)}\bigl((1+L)^s-1\bigr),
\end{equation}
with empty products equal to $1$.

\begin{samepage}
\begin{maintheorem}
Let $\X$ be a Hausdorff second-countable smooth orbifold without boundary,
possibly ineffective. If $\dim\X\leq n$ and $|G_x|\leq B$ for every full
stabilizer, there is a smooth Hermitian complex bundle $E$ of rank $R(n,B)$
such that, as a $G_x$-representation for every $x\in\X$,
\[
 E_x\cong\C[G_x]^{\oplus R(n,B)/|G_x|}.
\]
\end{maintheorem}
\end{samepage}

\begin{maincorollary}
Under the same hypotheses, a Hausdorff second-countable smooth manifold $M$
admits a smooth proper locally free action of $U(R(n,B))$ and
\[
 \X\simeq[M/U(R(n,B))].
\]
The $U(R(n,B))$-stabilizer of a point of $M$ represented by a frame over $x$
is conjugate to the image of $G_x$ in its fibre representation, and $M$ is
compact if and only if $|\X|$ is compact.
\end{maincorollary}

We allow disconnected orbifolds and manifolds of locally constant dimension.
In the pure-dimensional case the frame manifold has dimension
$\dim\X+R(n,B)^2$; the formula holds componentwise in general.
Every compact smooth orbifold satisfies the two bounds, so
Corollary B gives presentability without an effectiveness assumption.

For finite-dimensional proper $G$-CW complexes, with $G$ discrete and
isotropy orders bounded, L\"uck and Oliver constructed equivariant bundles
with regular stabilizer fibres
\cite[Corollary 2.7]{LO}. Henriques and Metzler gave a sufficient
presentability criterion for noneffective orbifolds \cite{HM}.
Pardon established regular bundles for a broader class of orbispaces,
with rank depending only on bounds for dimension and isotropy order
\cite[Theorem~1.1 and p.~2047]{Pardon}. We give an independent construction
in the smooth category with the explicit rank
\eqref{eq:rank}.

We construct the bundle over the skeleta of a locally finite triangulation.
For $B\geq2$, regular representations normalized to the common rank $L$
extend over the two-skeleton. On the boundary of a $(2s+1)$-cell,
decomposition into irreducible representations of the cell group gives
ordinary multiplicity bundles on $S^{2s}$. Extending across the cell requires
these bundles to be trivial. Put $m=1+L$ and let $\psi^m$ denote the $m$th
Adams operation in complex $K$-theory. It fixes every stabilizer
representation and multiplies the reduced $K$-theory of $S^{2s}$ by $m^s$.
Thus $m^s\id-\psi^m$ cancels the reduced boundary classes and multiplies
each regular fibre by $m^s-1$.

To realize the resulting virtual class, we express it as $[A]-[D]$, where
$A,D$ are finite sums of tensor products of exterior powers of the bundle
already constructed. The multiplicity estimate and the connectivity of
complex Stiefel manifolds give an equivariant monomorphism
$D\hookrightarrow A$. Its quotient is an actual regular bundle, and its
boundary multiplicity bundles have sufficient rank for stable triviality
to imply triviality. At the next even-dimensional stage, reduced $K^0$ of
the boundary sphere vanishes, and the same estimate permits attachment
without increasing the rank. One factor for each odd dimension at least
three gives \eqref{eq:rank}, independently of the number of cells.
Equivariant smoothing preserves the stabilizer representations, using
the argument of \cite[proof of Corollary~1.3]{Pardon}, and unitary frames
give the quotient presentation.

In Part~II \cite{BLII}, we optimize the integral Adams multipliers at each
attachment to obtain finer bounds for regular rank. Using Pardon's
representable approximation theorem, we extend these bounds to paracompact
orbispaces with bounded dimension and stabilizer orders. We introduce a
stable regular index and compute it for banded
cyclic gerbes over finite connected CW complexes through twisted $K$-theory.
We also obtain sharp low-dimensional bounds for stabilizers of order at most
two.

The same integral formulas in exterior powers define operations on
holomorphic and algebraic bundles over $\C$. We will explore elsewhere under what hypotheses the
corresponding virtual classes admit regular-bundle realizations in those
categories.

\section{Locally finite cells}\label{sec:cells}

An orbifold here is a smooth proper \'etale differentiable stack with
Hausdorff second-countable coarse space. Chart domains have no boundary;
reflection strata in the coarse space are allowed. We allow noncompact,
disconnected, and ineffective orbifolds. Write $BG=[*/G]$. In this section,
$\Vect$ denotes continuous complex bundles of locally finite rank, with maps
covering the identity. Skeletal restrictions are topological stacks.
Throughout Sections~\ref{sec:cells} and~\ref{sec:adams}, assume
$\dim\X\leq n<\infty$.

The effective reduction of $\X$ has a locally finite good triangulation
subordinate to any atlas: closed stars lie in chart footprints, simplices
lift homeomorphically, and the lifts form compatible finite-group
triangulations \cite[Proposition 1.2.3 and Lemmas 1.2.2, 1.2.4]{MP}. Use
this triangulation on each connected component, with the chart groups and
arrows of the original orbifold $\X$.
Since the triangulation is locally finite, second countability makes its
set of simplices countable. The construction below uses local finiteness.

\begin{theorem}[Bundle gluing]\label{thm:cells}
There are a countable locally finite triangulation $T$ of $|\X|$ of
dimension at most $n$, closed substacks
\[
 \X^j=\X\mathbin{\times}_{|\X|}|T^j|,\qquad \X^{-1}=\varnothing,
\]
and representable characteristic maps
\[
 \phi_\sigma:\sigma\times BG_\sigma\longrightarrow\X^j
 \qquad(\dim\sigma=j)
\]
with the following properties.
\begin{enumerate}[label=(\roman*)]
\item $G_\sigma$ acts trivially on $\sigma$ and is the full stabilizer on its
relative interior; $\phi_\sigma$ is an equivalence there and sends the
boundary to $\X^{j-1}$.
\item Bundles on $\X^j$ are equivalent to a bundle on $\X^{j-1}$, a
$G_\sigma$-bundle on each $j$-simplex, and specified boundary
isomorphisms. Compatible bundle maps glue.
\item The equivalence preserves direct sums, tensor products, exterior
powers, and every original stabilizer action.
\end{enumerate}
No uniform bound on $|G_x|$ is required.
\end{theorem}

The category in (ii) is
\begin{equation}\label{eq:gluing}
 \Vect(\X^{j-1})
 \mathbin{\times^{(2)}}_{\prod_{\dim\sigma=j}
       \Vect(\partial\sigma\times BG_\sigma)}
 \prod_{\dim\sigma=j}\Vect(\sigma\times BG_\sigma),
\end{equation}
where the superscript records the boundary isomorphisms. The equivalence
follows from the closed-cell clutching in Lemma~\ref{lem:clutching}.

\subsection{Characteristic disks and changes of chart}

Choose a uniformizer $[U/Q]\hookrightarrow\X$. Dividing $Q$ by the kernel of
its action gives a chart of the effective reduction. For each closed simplex
$\sigma$, choose a lift $\widetilde\sigma$ in an original chart and let
$G_\sigma$ be its setwise stabilizer. The coarse projection is injective on
the lift, hence $G_\sigma$ fixes it pointwise. Conversely, an element fixing
an interior point preserves the corresponding lifted interior simplex and
therefore its closure. Thus $G_\sigma$ is exactly the full interior
stabilizer, including the ineffective kernel. The map
\begin{equation}\label{eq:characteristic}
 [\widetilde\sigma/G_\sigma]=\sigma\times BG_\sigma
 \longrightarrow[U_\sigma/Q_\sigma]\longrightarrow\X
\end{equation}
is representable because its stabilizer maps are injective. Over the
interior, the translates of the lift are disjoint and exhaust the chart
preimage, so it is an equivalence there.

\begin{remark}\label{rem:boundary}
The full restriction $\X\mathbin{\times}_{|\X|}\sigma$ need not be
$\sigma\times BG_\sigma$: stabilizers may increase on the boundary.
Clutching uses the boundary restriction of the characteristic map.
\end{remark}

\begin{lemma}\label{lem:chart}
If $\sigma$ lies in the footprint of another uniformizer $[U/Q]$, the
pullback of \eqref{eq:characteristic} to $U$ is
$Q\times_{G_\sigma}\sigma$, for an injection $G_\sigma\to Q$, and maps to
$U$ as the induced family of lifted simplices.
\end{lemma}

\begin{proof}
Pull the principal $Q$-bundle $U\to[U/Q]$ back first to
$\sigma\times BG_\sigma$ and then to $\sigma$. It is trivial because
$\sigma$ is contractible. The commuting $G_\sigma$-action is therefore a
continuous map $\sigma\to\operatorname{Hom}(G_\sigma,Q)$, hence constant.
Representability makes its value injective. Write this homomorphism as
$\theta$. The quotient is
$Q\times_{G_\sigma}\sigma$, with
$(q,z)\sim(q\theta(h)^{-1},hz)$. One sheet maps bijectively to a lift of
$\sigma$; compactness of the domain and the Hausdorff property make this a
homeomorphism. Its interior is one lifted interior component, and the other
sheets are its $Q$-translates. Changing the trivialization only conjugates
$\theta$.
\end{proof}

\subsection{Equivariant clutching and descent}

\begin{lemma}[Equivariant clutching]\label{lem:clutching}
Let a finite group $Q$ act cellularly without inversions and attach a
locally finite, possibly infinite family of cells by
\[
 Z=A\cup_{\coprod_e Q/H_e\times S^{j-1}}
       \coprod_e Q/H_e\times D^j.
\]
Restriction is an equivalence between $Q$-bundles on $Z$ and a $Q$-bundle on
$A$, $H_e$-bundles on the disks, and specified equivariant boundary
isomorphisms. The same holds for bundle maps and commutes with direct sums,
tensor products, and exterior powers.
\end{lemma}

\begin{proof}
Choose an $H_e$-equivariant trivialization of each disk bundle, then induce
it to $Q$ and identify its boundary with the restriction of the bundle on
$A$. On a collar of $S^{j-1}$ in $D^j$, radial pullback extends
the prescribed boundary isomorphism. On each collar, compose radial
retraction onto the boundary with the attaching map into $A$. The glued
collar bundle is the pullback of the bundle on $A$ under this map. This
gives bundle charts along the old skeleton; the disk-interior charts cover
the remaining points.
Local finiteness reduces the construction near each point to finitely many
cells. Compatible bundle maps descend through the same quotient topology.
The construction also works after restriction to an open subset, since locally
finite closed-set pasting applies to maps. Restriction recovers the data, and
the quotient map from the reconstructed bundle to the original one is an
isomorphism in the collar charts and off the attaching locus. Thus
restriction and clutching are quasi-inverse and preserve the stated operations.
\end{proof}

Let $\mathcal G_j$ be the category in \eqref{eq:gluing}. Restriction gives
$\mathsf{Res}_j:\Vect(\X^j)\to\mathcal G_j$. To construct its quasi-inverse,
cover $|\X|$ by open stars whose closed stars lie in original chart
footprints. In a chart $[U/Q]$, the closed-star preimage is a finite
$Q$-simplicial complex; Lemmas~\ref{lem:chart} and~\ref{lem:clutching} turn
an object of $\mathcal G_j$ into a $Q$-bundle on its open-star preimage.

Restriction is fully faithful there: the old skeleton and induced closed
cells form a locally finite closed cover, so compatible continuous maps glue
uniquely by closed-set pasting, and equivariance is checked on the same
pieces. The assertion remains valid after passage to an open subset.

Cover an overlap of two such open substacks by common chart refinements,
obtained by shrinking orbifold charts inside the overlap. Their maps to the
two original charts are taken over $\X$. Pull back both constructed bundles
together with their identifications with the prescribed data.
Lemma~\ref{lem:chart}, restricted to the pieces in each refinement, identifies
the characteristic-cell pullbacks. Since clutching commutes with restriction,
both bundles induce the same prescribed gluing data on each common
refinement. Full faithfulness gives the unique comparison isomorphism inducing these
identifications. On further refinements the comparisons agree by uniqueness,
so they descend to an isomorphism on the overlap. On a triple overlap, the
composite of two comparisons and the direct comparison induce the same maps
on the prescribed data; full faithfulness makes them equal.

Open descent defines
$\mathsf{Glue}_j:\mathcal G_j\to\Vect(\X^j)$, on objects and morphisms. The
chosen identifications give
$\mathsf{Res}_j\mathsf{Glue}_j\cong\id$. For a bundle $F$, the canonical
maps on the cell pieces give
$\mathsf{Glue}_j\mathsf{Res}_j(F)\cong F$. Hence the functors are
quasi-inverse.

The construction takes place over the original stack. In a presentation,
each arrow acts by the unique map glued from its actions on the cell pieces, so
$\rho(g_2)\rho(g_1)=\rho(g_2g_1)$. Arrows inducing the same germ remain
distinct. Both functors commute with the fibrewise operations and preserve
ineffective isotropy and its monodromy. This proves
Theorem~\ref{thm:cells}.\hfill$\square$

\begin{remark}[Metrics on the skeleta]\label{rem:metrics}
Averaging over finite chart groups and using a locally finite partition of
unity gives every continuous bundle an invariant Hermitian metric. Thus every
equivariant bundle monomorphism splits orthogonally.
\end{remark}

\section{Adams operations and regular bundles}\label{sec:adams}

For a skeleton $Z$, write $K^0_{\mathrm{vb}}(Z)$ for the Grothendieck group,
under direct sum, of continuous complex bundles of bounded finite rank on
$Z$. We use finite bundle expressions and their pullbacks, without an
excision or exactness assumption on these groups.

\begin{lemma}[Relative embeddings]\label{lem:embedding}
Let $Z$ be one of the skeleta above and $Y\subset Z$ a subcomplex. Let
$A,B$ be bundles on $Z$ with a specified monomorphism $B|_Y\to A|_Y$. On each cell outside $Y$, let
$a_\rho,b_\rho$ be the multiplicities of an irreducible $\rho$ in the
pullbacks of $A,B$. If $a_\rho\geq b_\rho$ and, on a $j$-cell,
\begin{equation}\label{eq:gapcondition}
 2(a_\rho-b_\rho)\geq j-1
 \quad\hbox{whenever }b_\rho>0.
\end{equation}
then the monomorphism extends over $Z$.
\end{lemma}

\begin{proof}
On a constant-group disk, isotypic decomposition reduces the problem to
trivial multiplicity bundles. The space of injective equivariant maps
retracts onto a product of $V_{b_\rho}(\C^{a_\rho})$, the complex Stiefel
manifolds of orthonormal frames. Each is
$2(a_\rho-b_\rho)$-connected: forgetting the last frame vector
has a sphere as fibre, and induction with the homotopy exact sequence gives
this connectivity range. Thus the boundary map extends by
\eqref{eq:gapcondition}; at vertices nonemptiness suffices. Using all
injective maps keeps the prescribed boundary value fixed.
Theorem~\ref{thm:cells} glues the extensions, and local finiteness gives
continuity.
\end{proof}

\begin{theorem}[Continuous regular bundles]\label{thm:regular}
If $|G_x|\leq B$ for every $x$, then $\X$ has a continuous complex bundle
$E$ of rank $R=R(n,B)$ with
\[
 E_x\cong\C[G_x]^{\oplus R/|G_x|}
 \qquad(x\in\X).
\]
\end{theorem}

\begin{proof}
If $B=1$, the trivial complex line has the required fibres. Assume $B\geq2$.
For every injection $J\to H$,
\begin{equation}\label{eq:regularrestriction}
 \Res^H_J\C[H]\cong\C[J]^{\oplus[H:J]},
\end{equation}
and this remains true after changing the injection by an automorphism. At a
vertex take $\C[G_v]^{\oplus L/|G_v|}$. Endpoint restrictions on an edge
agree, and on the boundary of a two-cell every isotypic multiplicity bundle
on $S^1$ is trivial. Hence the construction extends through dimension two
with rank $L$, using the actual characteristic boundary maps. This proves
the assertion when $n\leq2$.

Put $m=1+L$ and, for $s\geq1$, set
\begin{equation}\label{eq:adams}
 c_s=m^s-1,\qquad T_s=m^s\id-\psi^m.
\end{equation}
Every stabilizer exponent divides $L$. The character identity
$\chi_{\psi^m V}(g)=\chi_V(g^m)$, with $g^m=g$, shows that $\psi^m$ fixes
each stabilizer representation ring. On a characteristic sphere with group $J$,
isotypic decomposition expresses a bundle as a sum of ordinary multiplicity
bundles tensored with irreducible representations of the cell group. Hence
\[
 K^0_{\mathrm{vb}}(S^q\times BJ)\cong K^0(S^q)\otimes R(J),
\]
where $R(J)$ is the ordinary representation ring. The reduced ordinary
factor is $\mathbb Z\beta_s$ for $q=2s>0$ and vanishes for odd $q$.
The integral identity $\psi^m\beta_s=m^s\beta_s$ follows for $s=1$ from
$\beta_1=[L_1]-1$ and $\beta_1^2=0$, where $L_1$ is the Hopf line bundle on
$S^2$. External products and Bott periodicity give the general case
\cite[Proposition~2.21]{Hatcher}.

Write $r_0=L$ and $r_s=c_sr_{s-1}$. Suppose $1\leq s\leq t(n)$ and a
regular bundle $E$ of rank $r_{s-1}$ has been constructed on $\X^{2s}$.
Its pullback $F$ to the boundary of a $(2s+1)$-cell with group $J$ has class
\[
 [F]=[W]+\beta_s a,\qquad
 W=\C[J]^{\oplus r_{s-1}/|J|},\quad a\in R(J).
\]
Indeed, at a boundary point with stabilizer $H$, the actual characteristic
map induces an injection $J\to H$, and \eqref{eq:regularrestriction} gives
exactly $r_{s-1}/|J|$ regular copies. Thus
\begin{equation}\label{eq:model-cancellation}
 T_s[F]=c_s[W]+(m^s-m^s)\beta_s a=c_s[W].
\end{equation}

The Newton polynomial for $\psi^m$ is integral in the exterior powers.
Evaluate it on $E$ and separate the positive and negative monomials in
$T_s[E]$ to obtain actual bundles $A,D$ with
\[
 T_s[E]=[A]-[D].
\]
These are fixed finite sums of tensor products of exterior powers of $E$,
so their ranks are finite and constant. The construction commutes with
pullback and with all descent maps. Only these finite expressions are
needed on the skeleton; the Adams calculation takes place on the
characteristic spheres and in the stabilizer representation rings.

For an irreducible $\rho$ of a cell group $G$, let $a_\rho,d_\rho$ be its
multiplicities in $A,D$. Since $|G|\mid L\mid r_{s-1}$,
\begin{equation}\label{eq:gaps}
 a_\rho-d_\rho
   =c_s\frac{r_{s-1}}{|G|}\dim\rho
   \geq c_s\geq sL\geq2s.
\end{equation}
On a $j$-cell of $\X^{2s}$, this gives
$2(a_\rho-d_\rho)\geq4s\geq j-1$. Lemma~\ref{lem:embedding}, with
$Y=\varnothing$, supplies
$D\hookrightarrow A$, which splits by Remark~\ref{rem:metrics}.
The quotient $E^\flat=A/D$ is an actual bundle of rank $r_s$, with
\[
 [E^\flat]=T_s[E],\qquad E^\flat_x\cong E_x^{\oplus c_s}.
\]
The fibre isomorphism follows from equality of irreducible multiplicities.

By \eqref{eq:model-cancellation}, each boundary multiplicity bundle of
$E^\flat$ has the class of a trivial bundle in $K^0(S^{2s})$. Equality in
this group gives an isomorphism after adding an auxiliary bundle. Since
the sphere is compact, that auxiliary bundle has a complement in a finite
trivial bundle. Adding the complement gives stable triviality. The
multiplicity ranks are
\[
 u_\rho=\frac{r_s}{|J|}\dim\rho\geq c_s\geq2s.
\]
A stably trivial rank-$u$ complex bundle on $S^q$, $q\geq2$, is trivial
when $2u>q$:
the fibrations $U(u)\to U(u+1)\to S^{2u+1}$ make stabilization injective on
$\pi_{q-1}U(u)$ in this range. Since $2u_\rho\geq4s>2s$, the boundary
multiplicity bundles are trivial. Choose equivariant boundary
trivializations and attach the constant regular bundles over the
$(2s+1)$-disks. Theorem~\ref{thm:cells} glues them to $E^\flat$, giving a
regular bundle of rank $r_s$ on $\X^{2s+1}$.

For the next attachment, pull this bundle back to $S^{2s+1}\times BJ$.
Its multiplicity ranks again satisfy $u_\rho=(r_s/|J|)\dim\rho\geq2s$,
and $\widetilde K^0(S^{2s+1})=0$ makes the multiplicity bundles stably
trivial. Since
\[
 2u_\rho\geq4s>2s+1\qquad(s\geq1),
\]
they are trivial. Attach the constant regular bundles over the
$(2s+2)$-disks, with no change in rank.

Continue for $s=1,\ldots,t(n)$, omitting the last even-dimensional
attachment if $n$ is odd. The resulting rank is
$r_{t(n)}=L\prod_{s=1}^{t(n)}c_s=R(n,B)$. On components of smaller
dimension, continue applying the polynomial operations through the remaining
stages, even when no new cells are attached. Thus every component has rank
$R(n,B)$. At each level, local finiteness supplies gluing,
and the rank multiplier depends only on the dimension and isotropy bounds.
\end{proof}

\begin{remark}
The bound \eqref{eq:rank} is uniform over any number of cells and components
with the same $n$ and $B$. For example,
$R(3,2)=R(4,2)=4$ and $R(5,2)=R(6,2)=32$.
For sharper bounds and obstructions to smaller regular rank, see
Part~II \cite{BLII}.
\end{remark}

\section{Representation-preserving smoothing}\label{sec:smoothing}

We follow Pardon's equivariant smoothing argument
\cite[proof of Corollary~1.3]{Pardon}. In the gluing step, the correcting
automorphism extends continuously by the identity across the boundary of
the overlap. No common bound on dimension or stabilizer order is needed.

\begin{samepage}
\begin{proposition}\label{prop:smoothing}
Let $E^0$ be a continuous complex bundle of constant finite rank $r$ on a
Hausdorff second-countable smooth orbifold $\X$. There are an invariant
continuous metric, a smooth Hermitian bundle $E^\infty$, and a continuous
equivariant unitary isomorphism
\[
 \Phi:E^0\longrightarrow E^\infty.
\]
Thus every full-stabilizer representation is preserved.
\end{proposition}
\end{samepage}

\begin{proof}
The zero bundle is already smooth. The coarse space and its open subsets are
metrizable and paracompact. Averaging chartwise bump functions and
normalizing a locally finite sum gives smooth orbifold partitions of unity.
Averaging local metrics and patching gives invariant smooth metrics on
bundles that are already smooth. Open sets below are open substacks;
equivariance refers to the full groupoid action.

Choose a chart $[U/G]$ centred at a $G$-fixed point $u$, and put $V=E^0_u$
with its full representation $\rho$. After shrinking, choose a continuous
trivialization $t_x:V\to E^0_x$ with $t_u=\id$. Write
$\alpha_g(y):E^0_y\to E^0_{gy}$ for the original action and set
\[
 T_x=\frac1{|G|}\sum_{g\in G}
       \alpha_g(g^{-1}x)t_{g^{-1}x}\rho(g)^{-1}.
\]
Then $T_u=\id$ and
$T_{\ell x}=\alpha_\ell(x)T_x\rho(\ell)^{-1}$.
Thus $T$ is an equivariant trivialization after an invariant shrinking.
Transporting the product smooth structure on $U\times V$ gives a local
smooth structure on $E^0$, with the original action of the full group $G$.

For smooth orbifold bundles $F,H$ with invariant smooth metrics, every
continuous equivariant map $h:F\to H$
admits a smooth equivariant approximation $k$ with
$\|k_x-h_x\|<\eta(x)$, for any positive continuous function $\eta$ on the
coarse space. Indeed, in a centred chart choose a smooth local section of
$\operatorname{Hom}(F,H)$ with value $h_u$ at the centre, and average over
the full chart group. Since $h_u$ intertwines the stabilizer actions,
averaging fixes this value. An invariant shrinking makes the error less
than $\eta/2$. A locally finite smooth partition of unity combines these
local sections with the same error bound. The support of each partition
function is closed in the ambient open suborbifold and contained in the
domain of its local section, so
each term extends smoothly by zero.

Suppose $E^0$ has smooth structures $E_A,E_B$ over open sets $A,B$. Put
$W=A\cap B$; the disjoint case is immediate. Let
$h:E_A|_W\to E_B|_W$ be the identity of the underlying continuous bundle.
Choose invariant smooth metrics on $E_A,E_B$ and a continuous function
$\varepsilon:|B|\to[0,1/4]$ positive on $|W|$ and zero on
$|B|\setminus|A|$. For a compatible metric on $|B|$, take the distance to this complement,
truncated at $1/4$; for an empty complement take $\varepsilon=1/4$.
Approximate $h$ on $W$ so that
\begin{equation}\label{eq:smoothing-relative}
 \begin{aligned}
 \|k_x-h_x\|&<\frac{\varepsilon(x)}
                   {2\max\{1,\|h_x^{-1}\|\}},\\
 \|k_xh_x^{-1}-\id\|&\leq\|k_x-h_x\|\,\|h_x^{-1}\|
             <\frac{\varepsilon(x)}2\leq\frac18.
 \end{aligned}
\end{equation}
Thus $k$ is a smooth isomorphism. Define
\begin{equation}\label{eq:smoothing-correction}
 a_x=\begin{cases}
       k_xh_x^{-1},&x\in W,\\
       \id,&x\in B\setminus A.
      \end{cases}
\end{equation}
The norm of $k_xh_x^{-1}-\id$ in \eqref{eq:smoothing-relative} is taken in
$E_B$, whose metric is defined throughout $B$. Since $\varepsilon$ vanishes
on the complement,
$a$ tends to the identity at every boundary point of $W$. It is therefore
a continuous equivariant automorphism on $B$, with continuous inverse by
matrix inversion.

Pull back the smooth structure of $E_B$ by $a$: an old smooth local frame
$e_1,\ldots,e_r$ is replaced by $a^{-1}e_1,\ldots,a^{-1}e_r$. This keeps
the underlying topology. For every groupoid arrow $\gamma$ over $B$, equivariance
gives
\[
 a_{t\gamma}\alpha_\gamma a_{s\gamma}^{-1}=\alpha_\gamma,
\]
so the original groupoid action is smooth in the new structure. On $W$,
$a h=k$ makes the identity $h$ a smooth isomorphism from $E_A$ to the new
$E_B$. On the orbifold site their smooth-section sheaves agree on the overlap,
and hence glue to a smooth structure over $A\cup B$ which leaves $E_A$
unchanged.

Choose a countable cover $\{B_i\}$ with the local smooth structures above.
Successive pairwise gluings give smooth-section sheaves $\mathcal S_j$ on
$A_j=\bigcup_{i\leq j}B_i$ satisfying
\begin{equation}\label{eq:smoothing-stationary}
 \mathcal S_{j+1}|_{A_j}=\mathcal S_j.
\end{equation}
Every point has an open neighbourhood whose smooth structure is fixed
after a finite step. The same holds near every arrow, since the sets $A_j$
are saturated. These sheaves therefore define a smooth bundle $E^\infty$
with the original topology and groupoid action. Choose an invariant smooth Hermitian metric
on $E^\infty$ and give $E^0$ the same metric. The identity of the underlying
continuous bundle is the required equivariant unitary isomorphism.
\end{proof}

\section{Frames and compactness}\label{sec:frames}

Proposition~\ref{prop:smoothing} applied to Theorem~\ref{thm:regular} proves
Theorem A. Put $R=R(n,B)$. Since regular representations are faithful, the
standard frame construction applies; see \cite[\S4]{Kalisnik}.

Choose a Hausdorff proper \'etale presentation
$\Gamma\rightrightarrows U$ associated to a countable cover by uniformizers.
The atlas $U$ is second-countable, as is its frame bundle $P=\Fr_U(E)$,
which is locally trivial with fibre $U(R)$. The commuting actions are
\[
 g\cdot f=\rho_g\circ f,
 \qquad f\cdot A=f\circ A,
 \qquad f:\C^R\xrightarrow{\cong}E_{s(g)}.
\]
The $\Gamma$-action is free: an arrow fixing a frame acts trivially on its
fibre and hence is the identity. To check properness, consider the action map
\[
 \Gamma\mathbin{{}_s\!\times_{\pi_P}}P\longrightarrow P\times P,
 \qquad(g,f)\longmapsto(g\cdot f,f).
\]
The inverse image of a compact set has $f$ in a compact subset of $P$ and
$(s(g),t(g))$ in the compact image under the two base projections. Properness
of $(s,t)$ implies that $g$ lies in a compact subset of $\Gamma$, and the inverse
image is closed in the resulting compact product.

Thus $\Gamma\ltimes P$ is free, proper, and \'etale. Its quotient
$M=\Gamma\backslash P$ is a Hausdorff smooth manifold, locally the quotient
of a frame chart by a free finite-group action. The quotient map is open, so
$M$ is second-countable. The right $U(R)$-action descends and is proper because
$U(R)$ is compact. If $m=[f]$ lies over $x$,
then
\begin{equation}\label{eq:framestabilizer}
 U(R)_m=\{\,f^{-1}\rho_gf:g\in G_x\,\}.
\end{equation}
Indeed, $[fA]=[f]$ exactly when $fA=\rho_gf$ for a unique isotropy arrow $g$.
The action is therefore locally free. The two commuting actions make $P$ a Morita bibundle between $\Gamma$ and
$U(R)\ltimes M$: $P\to U$ is the ordinary principal frame bundle and
$P\to M$ is principal for $\Gamma$. Hence
\[
 \X\simeq[M/U(R)]
\]
with full stabilizers identified by \eqref{eq:framestabilizer}. The quotient
map of a compact-group action on a locally compact Hausdorff space is closed
with compact fibres, hence proper. Therefore $M$ is compact exactly when
$M/U(R)=|\X|$ is compact. This proves Corollary B.

\begin{corollary}[Compact orbifolds]
Every compact smooth orbifold without boundary, possibly ineffective, has a
smooth regular bundle and is a quotient of a compact smooth manifold by a
locally free action of a unitary group.
\end{corollary}

\begin{proof}
A finite uniformizer cover bounds the stabilizer orders and component
dimensions; apply Theorem A and Corollary B.
\end{proof}

For example, on $BH$ a faithful representation
$\rho:H\hookrightarrow U(R)$ gives $M=\rho(H)\backslash U(R)$, whose right
stabilizer at the identity coset is $\rho(H)$; thus
$[M/U(R)]\simeq BH$. The chart group acts freely on frames. The induced
compact-group action on $M$ has stabilizers isomorphic to $H$.

A fibre of a regular bundle of rank $R$ is a positive multiple of
$\C[G_x]$, so $|G_x|\leq R$. Thus fixed finite regular rank requires
bounded stabilizer orders.

\section*{Acknowledgements}
This work forms part of Enrique Becerra's doctoral thesis at CINVESTAV, written
under the supervision of Ernesto Lupercio.

E.L. was supported by CONAHCYT grant CB-2017-2018-A1-S-30345, the Simons
Foundation under grant SFI-MPS-T-Institutes-00007697, and the Ministry of
Education and Science of the Republic of Bulgaria under grant
DO1-239/10.12.2024. Most of this work was completed during his sabbatical leave
from CINVESTAV\@. He thanks the International Center for Mathematical
Sciences--Sofia at the Institute of Mathematics and Informatics, Bulgarian
Academy of Sciences, and the Institute for the Mathematical Sciences of the
Americas at the University of Miami for hospitality and support.

\begin{samepage}
ChatGPT (OpenAI) was used to verify the arguments and the references,
improve the English text, and help prepare the LaTeX source.
The mathematical ideas are due to the authors.
The authors checked the final manuscript and take responsibility for it.
\par
\end{samepage}

\bibliographystyle{amsplain}
\bibliography{regular_bundles_on_orbifolds}
\end{document}